\documentclass{article}
\usepackage{latexsym,amssymb,amsthm,amsmath,cite,enumerate,algorithm2e,float,verbatim}
\newtheorem{theorem}{Theorem}[section]
\newtheorem{lemma}[theorem]{Lemma}

\usepackage{lineno}
\usepackage{setspace}

\begin{document}

\onehalfspace

\title{On some Graphs with a Unique Perfect Matching}

\author{S. Chaplick$^1$
\and M. F\"{u}rst$^2$
\and F. Maffray$^3$
\and D. Rautenbach$^2$}

\date{}

\maketitle

\begin{center}
$^1$ Lehrstuhl f\"{u}r Informatik I, Universit\"{a}t W\"{u}rzburg, Germany, \texttt{steven.chaplick@uni-wuerzburg.de}\\[3mm]
$^2$ Institut f\"{u}r Optimierung und Operations Research, Universit\"{a}t Ulm, Germany, \{\texttt{maximilian.fuerst,dieter.rautenbach}\}\texttt{@uni-ulm.de}\\[3mm]
$^3$ CNRS, Laboratoire G-SCOP, Universit\'{e} de Grenoble-Alpes, France,
\texttt{frederic.maffray@grenoble-inp.fr}
\end{center}

\begin{abstract}
We show that deciding whether a given graph $G$ of size $m$ has a unique perfect matching as well as finding that matching, if it exists, can be done in time $O(m)$ if $G$ is 
either a cograph, 
or a split graph,  
or an interval graph, 
or claw-free.
Furthermore, we provide a constructive characterization of the claw-free graphs with a unique perfect matching.
\end{abstract}

{\small 
\begin{tabular}{lp{13cm}}
\end{tabular}
}

\section{Introduction}

Bartha \cite{ba} conjectured that a unique perfect matching of a given graph $G$ of size $m$, if it exists, can always be found in $O(m)$ time.
Gabow, Kaplan, and Tarjan \cite{gakata} describe a $O(m\log^4m)$ algorithm for this problem. 
Furthermore, they show that, if apart from $G$, some perfect matching $M$ is also part of the input, 
then one can decide the uniqueness of $M$ in $O(m)$ time.
Since maximum matchings can be found in linear time for 
chordal bipartite graphs \cite{ch},
cocomparability graphs \cite{menini},
convex bipartite \cite{stye}, and
cographs \cite{copest,yuya},
also deciding whether these graphs have a unique perfect matching,
as well as finding the unique perfect matching, if it exists,
is possible in linear time.
Also for strongly chordal graphs given a strong elimination order \cite{daka},
a maximum matching can be found in linear time, and the same conclusion applies.
Levit and Mandrescu \cite{lema} showed that unique perfect matchings can be found in linear time for K\"{o}nig-Egerv\'{a}ry graphs and unicyclic graphs.

We contribute some structural and algorithmic results concerning graphs with a unique perfect matching.
First, we extend a result from \cite{gohile} to cographs and split graphs,
which leads to a very simple linear time algorithm deciding the existence of a unique perfect matching, and finding one, if it exists.
For interval graphs, we describe a linear time algorithm that determines a perfect matching, 
if the input graph has a unique perfect matching.
Similarly, for connected claw-free graphs of even order, we describe a linear time algorithm that determines a perfect matching.
Together with the result from \cite{gakata} this implies that for such graphs the existence of a unique perfect matching can be decided in linear time.
Finally, we give a constructive characterization of claw-free graphs with a unique perfect matching.

\section{Results}

For a graph $G$, we say that a set $U=\{ u_1,\ldots,u_k\}$ {\it forces a unique perfect matching in $G$}
if $n(G)=2k$, and $d_{G_i}(u_i)=1$ for every $i\in [k]$, where $G_i=G-\bigcup\limits_{j=1}^{i-1}N_G[u_i]$,
and $N_G[u]$ denotes the closed neighborhood of $u$ in $G$.
Clearly, if $U$ forces a unique perfect matching in $G$, then $G$ has a unique perfect matching 
$u_1v_1,\ldots,u_kv_k$, where $v_i$ is the only neighbor of $u_i$ in $G_i$ for $i\in [k]$.
As shown by Golumbic, Hirst, and Lewenstein (Theorem 3.1 in \cite{gohile}), 
a bipartite graph $G$ has a unique perfect matching 
if and only if some set forces a unique perfect matching in $G$;
their result actually implies that both partite sets of $G$ force a unique perfect matching.
This equivalence easily extends to cographs and split graphs.

\begin{theorem}
If $G$ is a cograph or a split graph, then $G$ has a unique perfect matching 
if and only if some set forces a unique perfect matching in $G$.
\end{theorem}
\begin{proof}
Since the sufficiency is obvious, we proceed to the proof of the necessity.
Therefore, let $G$ be a cograph or a split graph with a unique perfect matching $M$.
In view of an inductive argument, and since the classes of cographs and of split graphs are both hereditary,
it suffices to consider the case that $G$ is a connected graph of order at least $4$, and to show that $G$ has a vertex of degree $1$. 

First, suppose that $G$ is a cograph.
Since $G$ is connected, it is the join of two graphs $G_1$ and $G_2$.
If $G_1$ and $G_2$ both have order at least $2$, then $M$ contains 
either two edges between $V(G_1)$ and $V(G_2)$,
or one edge of $G_1$ as well as one edge of $G_2$.
In both cases, these two edges are part of an $M$-alternating cycle of length $4$, which is a contradiction.
Hence, we may assume that $V(G_2)$ contains exactly one vertex $v$, 
which is a universal vertex in $G$.
Let $u$ be such that $uv\in M$.
If $u'$ is a neighbor of $u$ in $G_1$, and $u'v'\in M$,
then $uu'v'vu$ is an $M$-alternating cycle of length $4$, which is a contradiction.
Hence, the vertex $u$ has degree $1$ in $G$.

Next, suppose that $G$ is a split graph.
Let $V(G)=S\cup C$, where $S$ is an independent set, and $C$ is a clique that is disjoint from $S$.
Since $G$ has a unique perfect matching, it follows easily that $|C|-|S|$ is either $0$ or $2$.
Since $G$ has order at least $4$, the set $S$ is not empty.
If no vertex in $S$ has degree $1$, then it follows, similarly as for bipartite graphs, that $G$ contains an $M$-alternating cycle,
which completes the proof. 
\end{proof}
If $G$ is given by neighborhood lists, then it is straightforward to decide the existence of a set that forces a unique perfect matching in $G$ in linear time,
by iteratively identifying a vertex of degree $1$, and removing this vertex together with its neighbor from $G$. 
Altogether, for a given cograph or split graph, one can decide in linear time whether it has a unique perfect matching, and also find that matching, if it exists.

Our next results concern interval graphs.

\begin{lemma}\label{lemma2}
Let $G$ be an interval graph with a unique perfect matching $M$, and 
let $\Big([\ell_u,r_u]\Big)_{u\in V(G)}$ be an interval representation of $G$ 
such that all $2n(G)$ endpoints of the intervals $[\ell_u,r_u]$ for $u\in V(G)$ are distinct.

If $u^*\in V(G)$ is such that $r_{u^*}=\min\{ r_u:u\in V(G)\}$, 
and
$v^*\in N_G(u^*)$ is such that $r_{v^*}=\min\{ r_v:v\in N_G(u^*)\}$,
then $u^*v^*\in M$.
\end{lemma}
\begin{proof}
Suppose, for a contradiction, that $u^*v\in M$ for some neighbor $v$ of $u^*$ that is distinct from $v^*$. Let $u$ be such that $uv^*\in M$.
By the choice of $u^*$ and $v^*$, and since the intervals $[\ell_u,r_u]$ and $[\ell_{v^*},r_{v^*}]$ intersect, also the intervals $[\ell_u,r_u]$ and $[\ell_v,r_v]$ intersect, that is, $uv\in E(G)$.
Now, $u^*vuv^*u^*$ is an $M$-alternating cycle of length $4$, 
which is a contradiction. 
\end{proof}
Since, for a given interval graph, 
an interval representation as in Lemma \ref{lemma2} 
can be found in linear time \cite{hamcpavi},
Lemma \ref{lemma2} yields a simple linear time algorithm 
to determine a perfect matching in a given interval graph $G$,
provided that $G$ has a unique perfect matching.

We proceed to claw-free graphs.

Let $G$ be a graph.
Let $P:u_1\ldots u_k$ be a path in $G$, where we consider $u_k$ to be the {\it last} vertex of $P$.
We consider two operations replacing $P$ with a longer path $P'$ in $G$.
\begin{itemize}
\item $P'$ arises by applying an {\it end-extension} to $P$, if $P'$ is the path $u_1\ldots u_kv$, where $v$ is some neighbor of $u_k$ that does not lie on $P$.
\item $P'$ arises by a {\it swap-extension} to $P$, if $k\geq 3$, and $P'$ is the path 
$$u_1\ldots u_{k-2}u_ku_{k-1}v,$$
where $v$ is some neighbor of $u_{k-1}$ that does not lie on $P$.
Note that $u_{k-2}$ and $u_k$ need to be adjacent for this operation.
\end{itemize}
The following lemma is a simple variation of a folklore proof of Sumner's result \cite{su} that connected claw-free graphs of even order have a perfect matching.
\begin{lemma}\label{lemma1}
If $G$ is a connected claw-free graph of even order, and $P:u_1\ldots u_k$ is a path in $G$ that does not allow an end-extension or a swap-extension, then the edge $u_{k-1}u_k$ belongs to some perfect matching of $G$.
\end{lemma}
\begin{proof}
In view of an inductive argument, 
it suffices to show that $G'=G-\{ u_{k-1},u_k\}$ is connected.
Suppose, for a contradiction, that $G'$ is not connected.
Clearly, $k\geq 2$.
If $k=2$, then $u_1$ has neighbors in two components of $G'$
while $u_2$ is only adjacent to $u_1$,
which yields a claw centered at $u_1$.
Now, let $k\geq 3$.
The path $u_1\ldots u_{k-2}$ lies in one component $K$ of $G'$.
Let $K'$ be a component of $G'$ that is distinct from $K$.
Since $P$ allows no end-extension, $u_k$ has no neighbor in $K'$.
Hence, $u_{k-1}$ has a neighbor $v$ in $K'$.
Since $u_{k-2}$ and $v$ are not adjacent, and $G$ is claw-free,
$u_k$ is adjacent to $u_{k-2}$, and $P$ allows a swap-extension,
which is a contradiction.
\end{proof}
Lemma \ref{lemma1} is the basis for the simple greedy algorithm \texttt{PMinCF} (cf. Algorithm \ref{alg1})
that determines a perfect matching in connected claw-free graphs of even order.

\begin{algorithm}[H]
\LinesNumbered\SetAlgoLined
\KwIn{A connected claw-free graph $G$ of even order.}
\KwOut{A perfect matching $M$ of $G$.}
\Begin{
$M\leftarrow \emptyset$; $k\leftarrow 0$\label{l2}\;
$\texttt{lm$\_$nb}(u)\leftarrow -1$ for all vertices $u$ of $G$\label{l4}\; 
\While{$n(G)\geq 2$\label{l17}}
{
\lIf{$k=0$}
{$k\leftarrow 2$; $u_1\leftarrow u$; $u_2\leftarrow v$, where $uv$ is some edge of $G$\label{l3}}
\Repeat{$\texttt{extend}=0$\label{l20}}
{\label{l18}
$\texttt{extend}\leftarrow 0$\label{l5}\;
\eIf{$\exists v\in N_G(u_k)\setminus \{ u_1,\ldots,u_k\}$\label{l6}}
{
$u_{k+1}\leftarrow v$; 
$k\leftarrow k+1$;
$\texttt{extend}\leftarrow 1$\label{l7}\;
}
{
\If{\rm $\texttt{lm$\_$nb}(u_k)=-1$\label{l8}}
{$\texttt{lm$\_$nb}(u_k)\leftarrow \max\{ \ell\in [k-1]:u_{k-i}\in N_G(u_k)\mbox{ for }i\in [\ell]\}$\label{l9}\;}
\If{\rm $\texttt{lm$\_$nb}(u_k)\geq 2$ and $\exists v\in N_G(u_{k-1})\setminus \{ u_1,\ldots,u_k\}$\label{l10}}
{$\texttt{lm$\_$nb}(u_k)\leftarrow\texttt{lm$\_$nb}(u_k)-1$\label{l13}\;
\If{\rm $\texttt{lm$\_$nb}(u_{k-1})\not=-1$\label{l14}}
{$\texttt{lm$\_$nb}(u_{k-1})\leftarrow\texttt{lm$\_$nb}(u_{k-1})+1\label{l15}$}
$x\leftarrow u_k$; 
$y\leftarrow u_{k-1}$\label{l11}\;
$u_{k-1}\leftarrow x$; 
$u_k\leftarrow y$;
$u_{k+1}\leftarrow v$\label{l12}\;
$k\leftarrow k+1$;
$\texttt{extend}\leftarrow 1$\label{l21}\;
}
}
}
$M\leftarrow M\cup \{ u_{k-1}u_k\}$;
$G\leftarrow G-\{ u_{k-1},u_k\}$;
$k\leftarrow k-2$\label{l16}\;
}\label{l19}
\Return $M$\;
}
\caption{\texttt{PMinCF}}\label{alg1}\end{algorithm}

\begin{theorem}\label{theorem1}
The algorithm \texttt{PMinCF} works correctly and can be implemented to run in $O(m(G))$ time for a given connected claw-free graph $G$ of even order.
\end{theorem}
\begin{proof}
Line \ref{l2} initializes the matching $M$ as empty and the order $k$ of the path $P:u_1\ldots u_k$ as $0$.
The {\bf while}-loop in lines \ref{l17} to \ref{l19} extends the matching iteratively as long as possible using the last edge $u_{k-1}u_k$ of the path $P:u_1\ldots u_k$.
If $k=0$, which happens in the first execution of the {\bf while}-loop, and possibly also in later executions, then, in line \ref{l3}, 
the path $P$ is reinitialized with $k=2$ using any edge $uv$ of $G$.
The {\bf repeat}-loop in lines \ref{l18} to \ref{l20} ensures that $P$ allows no end-extension and no swap-extension, which, by Lemma \ref{lemma1}, implies the correctness of \texttt{PMinCF}. 
The proof of Lemma \ref{lemma1} actually implies that $G$ stays connected throughout the execution of \texttt{PMinCF}.
In line \ref{l6} we check for the possibility of an end-extension,
which, if possible, is performed in line \ref{l7}.
If no end-extension is possible, we check for the possibility of a swap-extension.
The first time that some specific vertex $u_k$ is the last vertex of $P$,
and we check for the possibility of a swap-extension,
we set $\texttt{lm$\_$nb}(u_k)$ to the largest integer $\ell$ 
such that $u_k$ is adjacent to $u_{k-1},\ldots,u_{k-\ell}$.
Initializing $\texttt{lm$\_$nb}(u)$ as $-1$ for every vertex $u$ of $G$ in line \ref{l4}
indicates that its correct value has not yet been determined.
This happens for the first time in lines \ref{l8} and \ref{l9}.
Once $\texttt{lm$\_$nb}(u_k)$ has been determined,
it is only updated in line \ref{l13} for $u_k$,
and, if necessary, in line \ref{l14} for $u_{k-1}$.
Clearly, $u_k$ is adjacent to $u_{k-2}$ if and only if $\texttt{lm$\_$nb}(u_k)\geq 2$.
Therefore, line \ref{l10} correctly checks for the possibility of a swap-extension,
which, if possible, is performed in lines \ref{l11}, \ref{l12}, and \ref{l21}.
Altogether, the correctness follows, 
and it remains to consider the running time.

We assume that $G$ is given by neighborhood lists, that is, 
for every vertex $u$ of $G$, 
the elements of the neighborhood $N_G(u)$ of $u$ in $G$ 
are given as an (arbitrarily) ordered list.
Checking for the existence of a suitable vertex $v$ 
within the {\bf if}-statements in lines \ref{l6} and \ref{l10}
can be implemented in such a way that we 
traverse the neighborhood list of every vertex at most once
throughout the entire execution of \texttt{PMinCF}. 
Every time we check for the existence of such a neighbor $v$ of $u_k$,
we only need to consider the neighbors of $u_k$ 
that have not been considered before,
that is, we start with the first not yet considered neighbor of $u_k$ within its neighborhood list, and continue until we either find a suitable neighbor $v$ or reach the end of the list.
Since vertices that leave $P$ are also removed from $G$ in line \ref{l16},
this is correct, and the overall effort spent on checking for such neighbors is proportional to the sum of all vertex degrees, that is, $O(m(G))$.
The first computation of $\texttt{lm$\_$nb}(u_k)$ in line \ref{l9}
can easily be done in $O(d_G(u_k))$ time.
After that, every update of $\texttt{lm$\_$nb}(u_k)$ only requires constant effort.
Since $P$ is extended exactly $n(G)-2$ times,
the overall effort spent on maintaining $\texttt{lm$\_$nb}(u_k)$
is again proportional to the sum of all vertex degrees.
Altogether, it follows that the running time is $O(m(G))$,
which completes the proof.
\end{proof}
Again, it follows using \cite{gakata} that one can decide in linear time whether a given claw-free graph has a unique perfect matching.

Our final goal is a constructive characterization of the claw-free graphs 
that have a unique perfect matching
Let ${\cal G}$ be the class of graphs $G$ obtained by 
starting with $G$ equal to $K_2$,
and iteratively applying the following two operations:
\begin{itemize}
\item {\bf Operation 1}\\
Add to $G$ two new vertices $x$ and $y$, and the three new edges $xy$, $xu$, and $yu$,
where $u$ is a simplicial vertex of $G$.
\item {\bf Operation 2}\\
Add to $G$ two new vertices $x$ and $y$, the new edge $xy$,
and new edges between $x$ and all vertices in a set $C$,
where $C$ is a non-empty clique in $G$ such that 
$N_G(u)\setminus C$ is a clique for every vertex $u$ in $C$.
\end{itemize}

\begin{theorem}\label{theorem2}
A connected claw-free graph $G$ has a unique perfect matching if and only if $G\in {\cal G}$.
\end{theorem}
\begin{proof}
It is easy to prove inductively that all graphs in ${\cal G}$ are connected, claw-free, and have a unique perfect matching. Note that requiring $u$ to be simplicial in Operation 1 ensures that no induced claw is created by this operation. Similarly, the conditions imposed on $C$ in Operation 2 ensure that no induced claw is created.  

Now, let $G$ be a connected claw-free graph with a unique perfect matching $M$.
If $G$ has order $2$, then, trivially, $G$ is $K_2$, which lies in ${\cal G}$.
Now, let $G$ have order at least $4$.
By Kotzig's theorem \cite{ko,ne}, $G$ has a bridge that belongs to $M$.
In particular, $G$ is not $2$-connected.
Let $B$ be an endblock of $G$.
If $n(B)\leq 3$, then $B$ is $K_2$ or $K_3$, and the claw-freeness of $G$ easily implies that $G$ arises from a proper induced subgraph of $G$ by applying Operation 1 or 2.
Hence, we may assume that $n(B)\geq 4$.
If $n(B)$ is even, then, by Kotzig's theorem,
$B$, and hence also $G$, has two distinct perfect matchings, 
which is a contradiction.
Hence, $n(B)$ is odd, that is $n(B)\geq 5$.
If $u$ is the cutvertex of $G$ in $B$, then, since $G$ is claw-free, $N_B(u)$ is a clique of order at least $2$. This implies that $B-u$ is $2$-connected.
Again, by Kotzig's theorem, $B-u$, and hence also $G$, has two distinct perfect matchings, 
which is a contradiction, and completes the proof.
\end{proof}
Note that Theorem 3.4 in \cite{washyu}, and also Theorem 3.2 in \cite{washyu} restricted to claw-free graphs, follow very easily from Theorem \ref{theorem2} by an inductive argument.

\medskip

\noindent {\bf Acknowledgment} We thank Vadim Levit for drawing our attention to the problem studied in this paper.

\end{document}